\documentclass{amsart}
\usepackage[T1]{fontenc}
\usepackage[english]{babel}
\usepackage{amsmath,amsthm,amssymb, latexsym,geometry,stackrel,enumerate}
\usepackage[all]{xy}

\usepackage[dvips]{graphicx}
\DeclareGraphicsExtensions{.bmp,.png,.pdf,.jpgg}
\usepackage[all]{xy}
\usepackage{verbatim}
\usepackage[mathcal]{euscript}
\usepackage{epsfig}
\usepackage{multicol}
\usepackage{color}
\usepackage{verbatim}
\usepackage{graphicx,float}

\swapnumbers
\theoremstyle{plain}
\newtheorem{teo}{Theorem}[section]

\newtheorem{prop}[teo]{Proposition}

\theoremstyle{definition}

\newtheorem{ejem}[teo]{Example}

\newcommand{\ts}[1]{\normalfont{\textsf{#1}}}
\newcommand{\K}{\ts k}

\renewcommand{\a}{\alpha}
\renewcommand{\b}{\beta}

\newcommand{\G}{\Gamma}

\newcommand{\HH}[1]{\operatorname{dim}_{\K}\operatorname{HH}^{#1}(A)}
\newcommand{\Hh}[1]{\operatorname{dim}_{\K}\operatorname{HH}^{#1}(\Lambda)}
\newcommand{\abs}[1]{\lvert #1\rvert}

\newcommand{\Ext}[4]{\ts{Ext}_{#1}^{#2}(#3,\,#4)}

\newcommand{\rad}[1]{\ts{rad} #1}

\newcommand{\Hom}[3]{\ts{Hom}_{#1}(#2,\,#3)}

\newcommand{\mor}[3]{$#1\colon #2 \to #3$}

\newcommand{\Ge}[1]{\operatorname{dim}_{\K} \K \mathcal G_{#1}/ \operatorname{Im}(1-t)}

\newcommand{\Par}[1]{^-(0,0)^-_{#1}}

\title{Hochschild cohomology of $m$-Cluster tilted algebras of type $\widetilde{\mathbb{A}}$}

\author[V. Gubitosi]{Viviana Gubitosi}
\address{Instituto de Matem\'{a}tica y Estad\'{\i}stica Rafael Laguardia, Facultad de Ingenier\'{\i}a - UdelaR, Montevideo, Uruguay, 11200 }

\email{gubitosi@fing.edu.uy}


\begin{document}
\maketitle

\begin{abstract}
 In this paper, we compute the dimension of the Hochschild cohomology groups of any $m$-cluster tilted algebra of type $\tilde{\mathbb{A}}$. Moreover, we give conditions on the bounded quiver of an $m$-cluster tilted algebra $\Lambda$ of type $\tilde{\mathbb{A}}$ such that the Gerstenhaber algebra $\operatorname{HH}^*(\Lambda)$ has non-trivial multiplicative structures. We also show that the derived class of gentle $m$-cluster tilted algebras is not always completely determined by the dimension of the Hochschild cohomology.
\end{abstract}

\section*{Introduction}

The Hochschild cohomology groups $\operatorname{HH}^{i}(A)$ of an algebra $A$, where $i\geq 0$, were introduced by Hochschild in \cite{Ho46}. The
low-dimensional groups, namely for $i=0,1,2$, have a concrete interpretation of classical algebraic structures, but in general it is quite hard to compute
them. However, an explicit formula for the dimension of the Hochschild cohomology $\operatorname{HH}^{i}(A)$ of some subclasses of
special biserial algebras had been computed in terms of combinatorial data, for example in \cite{B06,ST10,Lad12b}. Recently Redondo and Rom\'an obtained a formula for quadratic string algebras and therefore for gentle algebras (see \cite{RR15}). Later,  Valdivieso in \cite{Va} compute the dimension of the Hochschild cohomology groups for Jacobian algebras from unpuctured Riemann surfaces.

Cluster categories were introduced in \cite{Buan2006} as a representation theoretic framework for the  cluster algebras of Fomin and Zelevinski in \cite{FZ02}. Several consequences were given by this categorification, for example, there is a bijection between the set of clusters in a cluster algebra and the set of the basic tilting objects in the corresponding cluster category, see \cite{Buan2006}.
Given an hereditary finite dimensional algebra $H$  over  an algebraically closed field $\K$ the $m$-cluster category is defined to be $\mathcal{C}_m(H):=\mathcal{D}^b(H)/ \tau^{-1} [m]$, where $[m]$  denotes the $m$-th power of  the shift functor $[1]$ and $\tau$ is the Auslander - Reiten translation in $\mathcal{D}^b(H)$. By a result of Keller \cite{K05}, the $m$-cluster category is triangulated.  For the  $m$-cluster category,  $m$-cluster tilting objects have been defined by Thomas, in \cite{Thomas2007}, who in addition showed that they are in bijective correspondence with the $m$-clusters associated by  Fomin and Reading to a finite root system in \cite{FR05}. The endomorphism algebras of the $m$-cluster tilting objects are called $m$-cluster tilted algebras or, in case $m=1$, cluster tilted algebras.

In \cite{ABCP09} it has been shown that cluster tilted algebras are gentle if and only if they are of type  $\mathbb{A}$ or $\tilde{\mathbb{A}}$. On the other hand, using arguments similar to those of \cite{ABCP09}, Murphy showed in \cite{Murphy2010} that $m$-cluster tilted algebras of type $\mathbb{A}$ are gentle and he described the connected components of $m$-cluster tilted algebras up to derived equivalence, a result analogous to that of \cite{Buan2008}. Later, a similar work has been done in \cite{Gubitosi} for  $m$-cluster tilted algebras of type $\mathbb{\widetilde{A}}$, where it is shown that $m$-cluster tilted algebras of type $\mathbb{\widetilde{A}}$ are gentle and their possible bound quivers are described. Moreover, in \cite{BG14} and \cite{Gub1} the algebras that are derived equivalent to $m$-cluster tilted algebras of type  $\mathbb{A}$ and $\mathbb{\widetilde{A}}$ have been classified. They are called $\mathbb{A}$-branched algebras   \cite[Definition 4.3]{BG14} and  $\mathbb{\widetilde{A}}$-branched algebras \cite[Definition 3.2 ]{Gub1} respectively.

The aim of this work is to compute the dimension of $\operatorname{HH}^{n}(\Lambda)$ for any $m$-cluster tilted algebra $\Lambda$ of type $\tilde{\mathbb{A}}$. Since the  Hochschild cohomology is a derived invariant \cite{Ric91} we can extend the result to  the class  of $\tilde{\mathbb{A}}$-branched algebras.  Moreover, using the results of Redondo and Rom\'an we obtain that the non-trivial multiplicative structure of the Gerstenhaber algebra $\rm HH^*(\Lambda)$ for  $\Lambda$  an $m$-cluster tilted algebra  of type $\tilde{\mathbb{A}}$ (or more generally  an  $\tilde{\mathbb{A}}$-branched algebra) depends on the existence of  $m$-saturated cycles.

We now state the main results of this paper (for the definitions of the terms used, we refer the reader to section 3  below).

\subsection*{Theorem A}\textit{ Let $\Lambda$ be an   $\tilde{\mathbb{A}}$-branched algebra with parameters $s_1,s_2,k_1,k_2,r$ and  $m\geq 1$.   Then:}

\begin{itemize}
\item[(a)] $\Hh{0}= \begin{cases} 2 & \mbox{ if \ }  r=1,  \  k_1=s_1=0 \text{ or  }  r=-1,  \  k_2=s_2=0;
\\
 1  & \mbox{otherwise}. \end{cases}$

\item[(b)] $\Hh{1}= \begin{cases} 3 & \mbox{ if \ }  k_1=k_2=0, \  s_1=s_2=1 ;
\\
 k_1+k_2+2  & \mbox{if \ }  r=0,  \  s_1=1,  \  k_1=0 \text{ or  }  r=0,  \  s_2=1,  \  k_2=0 ; \\
k_1+k_2 + 1 & \mbox{otherwise}. \end{cases}$

\item[(c)] If $\operatorname{char}\K\neq 2$ and $n\geq 2$, then $$\Hh{n}= \begin{cases}  k_1+k_2 & \mbox{ if \ } n\equiv 0,1 \pmod{\rm lcm(m+2,2)}; \\ 0 &
\mbox{otherwise.} \end{cases}$$

\item[(d)] If $\operatorname{char}\K= 2$ and $n\geq 2$, then $$\Hh{n}= \begin{cases} k_1+k_2 & \mbox{ if \ } n\equiv 0,1 \pmod{m+2};\\ 0 & \mbox{otherwise.} \end{cases}.$$
\end{itemize}


Moreover, we have:

\subsection*{Theorem B}\textit{ Let $\Lambda$ be an  $m$-cluster tilted algebra of type $\tilde{\mathbb{A}}$.
If its bound quiver contains at least one $m$-saturated cycle,
then the cup product defined in $\operatorname{HH}^*(\Lambda)$ is non trivial; and if $\operatorname{char}\K=0$, then the Lie bracket is also
non trivial.}
\medskip

In particular, specializing to the case $m=1$, we recover a particular case of  known results of \cite{Va}.\\

The paper is organized as follows: In section 1 we recall facts about gentle algebras, Hochschild cohomology, $m$-cluster tilted algebras and $\tilde{\mathbb{A}}$-branched algebras. Also  we establish the facts about $m$-cluster tilted algebras of type $\mathbb{\widetilde{A}}$  that will be used in the sequel. In section 2 we recall facts about the dimension of the Hochschild cohomology groups of gentle algebras. Section 3  and 4 are devoted to the proof of the Theorem  A and B respectively. In section 5 we show that the dimension of the Hochschild cohomology  is not a complete invariant for $\tilde{\mathbb{A}}$-branched algebras.\\

\section{Preliminaries}

\subsection{Gentle algebras}
While we briefly recall some  concepts concerning bound quivers and algebras, we refer the reader to \cite{ASS06} or \cite{ARS95}, for instance, for unexplained notions.

Let \K \  be an algebraic closed field. A quiver $Q$ is the data of two sets, $Q_0$ (the \textit{vertices}) and $Q_1$ (the \textit{arrows}) and two maps \mor{s,t}{Q_1}{Q_0} that assign to each arrow $\a$ its \textit{source} $s(\a)$ and its \textit{target} $t(\a)$. We write \mor{\a}{s(\a)}{t(\a)}. If $\a,\b \in Q_1$ are such that $t(\a)=s(\b)$ then the composition of $\a$ and $\b$ is the path $\a\b$. This extends naturally to paths of arbitrary positive length. The \emph{path algebra} $\K Q$ is the $\K$-algebra whose basis is the set of all paths in $Q$, including one stationary path $e_x$ at each vertex $x\in Q_0$, endowed with the  multiplication induced from the composition of paths. In case $|Q_0|$ is finite, the sum of the stationary paths  - one for each vertex - is the identity.

If the quiver $Q$ has no oriented cycles, it is called \emph{acyclic}. A \emph{relation} in $Q$ is a $\K$-linear combination of paths of length at least $2$ sharing source and target.  A relation which is a path is called \emph{monomial}, and the relation is \emph{quadratic} if the paths appearing in it have all length $2$. Let $\mathcal{R}$ be a set of relations.
 Given $\mathcal{R}$ one can consider the two-sided ideal of $\K Q$ it generates $I=\langle \mathcal{R}\rangle \subseteq  \langle Q_1 \rangle^2$. It is called \emph{admissible} if there exists a natural number $r\geqslant 2$ such that $\left\langle Q_1 \right\rangle^r \subseteq I$. The pair $(Q,I)$ is a \emph{bound quiver}, and $A=\K Q/I$ is its associated algebra. We said $\K Q/I$ is quadratic monomial if $I$ is generated by paths of length two.
It is known that any finite dimensional basic algebra over an algebraically closed field is obtained in this way, see \cite{G72}.

The class of gentle algebras  defined by Assem and Skowro\'nski in \cite{AH81}  has been extensively studied, see \cite{AS87, AG08, BB10, Buan2008, Murphy2010, SZ03}, for instance, and is particularly well understood, at least from the representation theoretic point of view. This class includes, among others, iterated tilted,  cluster tilted and $m$-cluster tilted algebras of types $\mathbb{A}$ and $\tilde{\mathbb{A}}$, and, as shown in \cite{SZ03}, is closed under derived equivalence.

Recall  that an algebra  $A= \K Q/I$ is said to be \emph{gentle} if  $I$ is generated by a set of monomial quadratic relations such that:

\begin{enumerate}
 \item[G1.] For every vertex $x\in Q_0$  at most two arrows enter or leave $x$;
 \item[G2.] For every arrow $\alpha\in Q_1$ there exists at most one arrow $\beta$ and one arrow $\gamma$ in $ Q_1$ such that $\alpha\beta\not\in I$, $\gamma\alpha\not\in I $;
 \item[G3.] For every arrow $\alpha\in Q_1$ there exists at most one arrow $\beta$ and one arrow $\gamma$ in $Q_1$ such that $\alpha\beta\in I$, $\gamma\alpha\in I$.
\end{enumerate}


\subsection{$m$-Cluster tilted algebras of type $\widetilde{\mathbb{A}}$}

Since $m$-cluster tilted algebras of type $\widetilde{\mathbb{A}}$ are gentle, and in particular, quadratic monomial, we assume that $\K Q/I$ is quadratic monomial, unless otherwise stated. Given a bound quiver $(Q,I)$ and an integer $m$, a cycle is called \emph{ $m$-saturated} if it is an oriented cycle consisting of $m+2$ arrows such that the composition of any two consecutive arrows on this cycle belongs to $I$. Recall that two relations  $r$ and $r'$ in the bound quiver $(Q,I)$ are said to be   \textit{consecutive} if there is a walk  $v=wr=r'w'$ in  $(Q,I)$ such that $r$ and $r'$ point in the same direction and share an arrow. Also, recall that given a connected quiver $Q$, its \emph{Euler characteristic} is  $\chi(Q)=|Q_1|-|Q_0|+1$.

Recall from \cite[Definition 7.2]{Gub1} that   $A\cong \K Q/I$ is an  \textit{algebra with root } if its bound quiver is gentle, connected, having exactly $\chi(Q)-1$  $m$-saturated cycles and no loops.

Since  $\chi(Q)$ is the number of $m$-saturated cycles plus 1, we know that  $(Q,I)$  has at least a non $m$-saturated cycle $\widetilde{\mathcal{C}}$. We will refer to the  cycle $\widetilde{\mathcal{C}}$ as a \textit{root cycle}. Moreover, since $A$ is a finite dimensional algebra, any  oriented  root cycle $\widetilde{\mathcal{C}}$ has at least one relation.

Once we fix a root cycle $\widetilde{\mathcal{C}}$, we can fix one drawing of it, i.e., one embedding into the plane. Thus, we can speak of clockwise and counterclockwise oriented arrows of $\widetilde{\mathcal{C}}$. But we have to consider that this notation is only unique up to reflection of the cycle, i.e., up to changing the roles of clockwise and counterclockwise oriented arrows.

Recall that for $m\geq 2$, an $m$-cluster tilted algebras of type $\widetilde{\mathbb{A}}$ does not have to be connected.

\subsection*{Theorem }\cite[Theorem 8.16]{Gubitosi}\textit{ A connected algebra $A=\K Q/I$ is  a connected component of an $m$-cluster tilted algebra of type $\widetilde{\mathbb{A}}$ if and only if  $(Q,I)$ is a gentle bound quiver satisfying the following conditions: }
\textit{\begin{itemize}
  \item [(a)] It can contain a non-saturated  cycle  $\widetilde{\mathcal{C}}$ in such a way that $A$ is an algebra with root $\widetilde{\mathcal{C}}$.
  \item [(b)] If it does not contain a non-saturated  cycle as in $(a)$, then the only  possible cycles  are  $m$-saturated.
  \item [(c)] Outside of an $m$-saturated cycle it can have at most  $m-1$ consecutive relations.
   \item [(d)] $r_h\equiv r_a$ modulo $m$.
\end{itemize}}

where $r_h$ is the \textit{number of clockwise internal relations} and $r_a$ is the \textit{number of counterclockwise internal relations}. See  \cite[Definition 8.13]{Gubitosi} for a complete definition of these numbers of relations on a root cycle.

The class of algebras derived equivalents to $m$-cluster tilted algebras of type $\widetilde{\mathbb{A}}$ with a root cycle were called in \cite{Gub1}  $\widetilde{\mathbb{A}}$-branched algebras.

\subsection{Hochschild cohomology}

Given an algebra $A$, the $n$-th Hochschild cohomology group of $A$ with coefficients in the bimodule $_A A_{A}$ is the extension group $\operatorname{HH}^{n}(A)=\Ext{A-A}{^n}{A}{A}$. The sum $\operatorname{HH}^{\ast}(A)= \bigoplus_{n\geqslant 0} \operatorname{HH}^{n}(A)$ has the additional structure of a Gerstenhaber algebra, see \cite{G63}. From \cite{R89, K04} this structure is known to be a derived invariant, that is, invariant under derived equivalence.

Let $A=\K Q /I$ be a monomial quadratic algebra. Define $\G_0=Q_0,\ \G_1=Q_1$, and for $n\geqslant 2$, $\G_n=\{\a_1\a_2\cdots \a_n|\ \a_i\a_{i+1}\in I\}$. Moreover, let $E=\ts{k}Q_0$ be the semi-simple algebra isomorphic to $A/\rad{A}$, and $\ts{k}\G_n$ the $\ts{k}$-vector space with basis $\G_n$. The latter are also $E-E$-bimodules in an obvious way. In what follows, tensor products are taken over $E$.

As mentioned before, the sum $\operatorname{HH}^{\ast}(A) = \bigoplus_{n\geqslant 0} \operatorname{HH}^{n}(A)$ has additional structure given by two products, which we now describe, see \cite{G63}. The two products are defined using the bar resolution $( A^{\otimes^n} , b_n)_{n\geq 0}$ of $A$.

Given $f\in \Hom{E-E}{A^{\otimes^{n}}}{A}$, $g\in\Hom{E-E}{A^{ \otimes^{m}}}{A}$, and $i\in\{1,2,\ldots n\}$, define the element $f\circ_i g$ such that $f \circ_i g(v_1 \otimes \cdots \otimes v_{n+m-1}) = f(v_1 \otimes \cdots \otimes v_{i-1} \otimes g( v_i \otimes \cdots \otimes v_{i+m-1} ) \otimes v_{i+m} \otimes
\ldots \otimes v_{n+m-1} )$.  In addition, define the \textit{composition product} as
$$f \circ g = \sum_{i=1}^n (-1)^{(i-1)(m-1)} f\circ_i g$$
and the \emph{bracket} to be $$[f,g]= f\circ g - (-1)^{(n-1)(m-1)}g \circ f$$

\medskip

On the other hand, denote by \mor{\sigma}{A\otimes A}{A} the multiplication of $A$. The \emph{cup-product} $f\cup g$ of $f$ and $g$ is  the  element of $\Hom{E-E}{A^{\otimes^{m+n}}}{A}$ defined by $f\cup g = \sigma(f \otimes g)$.

Recall that a Gerstenhaber algebra is a graded $\K$-vector space $A$ endowed with a product which makes $A$ into a graded commutative algebra, and a bracket $[ - ,  - ]$ of degree $1$ that makes $A$ into a graded Lie algebra, and such that $[x, yz] = [x, y]z + (-1)^{(|x|-1)|y|} y[x, z]$, that is, a graded analogous of a Poisson algebra. The cup product $\cup$ and the bracket $[ - , - ]$ defined above define products in $\operatorname{HH}^{\ast}(A)$, which becomes then a Gerstenhaber algebra (see \cite{G63}).

\section{ Hochschild cohomology groups of gentle algebras}

The dimension of the  Hochschild cohomology groups of gentle algebras have already been computed in \cite{Lad12b} by Ladkani and in \cite{RR15} by Redondo and Rom\'an. In the
first case,  these results have been expressed in terms of the derived invariant introduced by Avella-Alaminos and Geiss (AG-invariant for short) in \cite{AG08}. In the second one, Redondo and Rom\'an used Bardzell's resolution (see \cite{Ba97}).

The AG-invariant is a function $\phi_{A}:\mathbb N^2\rightarrow \mathbb N$ depending on the ordered pairs generated by a certain algorithm. The
number $\phi_{A}(n,m)$ counts how often each pair $(n,m)$ appears in the algorithm.  See \cite{AG08} for a complete definition of AG-invariant.

Since $m$-cluster tilted algebras of type $\tilde{\mathbb{A}}$  are gentle, we use the computation of Ladkani to prove our main result. The statement of Ladkani is the following:

\medskip

\begin{teo}\cite[Corollary 1]{Lad12b}\label{TeoLad} Let $A$ be a gentle algebra. Define $\psi_A(n)=\sum_{d\mid n} \phi_A(0,d)$ for $n\geq 1$. Then
\begin{itemize}
\item[(a)] $\Hh{0}=1+ \phi_A(1,0)$.
\item[(b)] $\Hh{1}=1+\abs{Q_1}-\abs{Q_0}+\phi_A(1,1)+ \begin{cases} \phi_A(0,1) & \mbox{if }
\operatorname{char}(k)=2\\ 0& \mbox{otherwise} \end{cases}$.
\item[(c)] $\Hh{n}=\phi_A(1,n)+a_n\psi_A(n)+b_n\psi_A(n-1)$ for $n\geq 2$, where $$(a_n,b_n)=
\begin{cases} (1,0)& \mbox{if } \operatorname{char}k\neq 2 \mbox{ and $n$ is even}\\ (0,1)&\mbox{if } \operatorname{char}k\neq 2 \mbox{ and $n$ is odd}\\
(1,1)&\mbox{if } \operatorname{char}k=2 \end{cases} $$
\end{itemize}

\end{teo}

\medskip


We include here an alternative proof of Theorem \ref{TeoLad} rewriting \cite[Corollary 1]{RR15} in terms of the invariant $\phi_A$. From now on, we follow the notation of \cite{RR15}.

\medskip

Let $A=\K Q/I$ be a gentle algebra,  we fix a set $\mathcal P$ of paths in $Q$ such that the set $\{\gamma+I \mid \gamma\in \mathcal P\}$ is a basis of $A$.  Then  the dimension of each Hochschild cohomology group  of $A$ can be computed using  the following  subsets of pairs of parallel paths.

\medskip

 Let $^-(Q_0 // \mathcal P_1)^- $ be the set of pairs $(e_r, \gamma)\in Q_0 \times \mathcal P$ such that $s(\gamma)=t(\gamma)=e_r$,  $|\gamma|\geq 1$, $Q_1\gamma\subset I$ and $\gamma Q_1\subset I$.

  Let  $^-(0,0)^-_1 $ be the set of pairs $(\alpha, \gamma)\in Q_1\times \mathcal P$ such that $s(\a)=s(\gamma)$, $t(\a)=t(\gamma)$,  $Q_1\gamma\subset I$ and $\gamma Q_1\subset I$.        

\medskip

For $n\geq 2$, write $\G_n=\{\alpha_1\alpha_2\cdots\alpha_n\mid  \alpha_i\alpha_{i+1}\in I\}$.

Let $^-(0,0)^-_n$ be the set of  pairs  $(\alpha_1\alpha_2\cdots\alpha_n, \gamma) \in \G_n\times \mathcal P$ such that $s(\alpha_1)=s(\gamma)$, $t(\alpha_n)=t(\gamma)$, $\gamma \notin \alpha_1 \K Q\cup \K Q\alpha_n$, $Q_1\gamma\subset I$ and $\gamma Q_1\subset I$.

  Let $\mathcal{C}_n$ be the set of  pairs  $(\alpha_1\alpha_2\cdots\alpha_n, e_r) \in \G_n\times Q_0$ such that $s(\alpha_1)=t(\alpha_n)=e_r$ and $\alpha_n\alpha_1\in I$. The pairs belonging to  any set $\mathcal{C}_n$ are called \textit{complete pairs}. Let $\mathcal{C}_n(0)$ be the subset of complete pairs $(\alpha_1\alpha_2\cdots\alpha_n, e_r) \in \G_n\times Q_0$ such that  it does not exist $\gamma\in Q_1\setminus\{\alpha_n\}$  nor $\beta\in Q_1\setminus\{\alpha_1\}$  with $\alpha_n\beta, \gamma\alpha_1\in I$.
The cyclic group $\mathbb{Z}_n=<t>$ of order $n$ acts on the set  $\mathcal{C}_n$ with the action given by $t(\alpha_1\alpha_2\cdots\alpha_n, e_{s(\alpha_1)})=(\alpha_n\alpha_1\alpha_2\cdots\alpha_{n-1}, e_{s(\alpha_n)})$. A complete pair $(\alpha_1\alpha_2\cdots\alpha_n, e_r)$ is called  \textit{gentle} if $t^m(\alpha_1\alpha_2\cdots\alpha_n, e_r)\in \mathcal{C}_n(0)$ for any $m\in \mathbb{Z}$.  We will denote by  $\mathcal G_{n}$  the set of gentle pairs in $\G_n\times Q_0$ and  by $\K \mathcal G_n$ the $\K$-vector space generated by the elements of $\mathcal G_n$.

In the same way let $\mathcal{I}_n$ be the set of  pairs  $(\alpha_1\alpha_2\cdots\alpha_n, e_r) \in \G_n\times Q_0$ such that $s(\alpha_1)=t(\alpha_n)=e_r$ and $\alpha_n\alpha_1\notin I$. The pairs belonging to  any set $\mathcal{I}_n$ are called \textit{incomplete pairs}. An incomplete pair $(\alpha_1\alpha_2\cdots\alpha_n, e_r)$ is called  \textit{empty} if there is no relation  $\beta\gamma\in I$  with $t(\beta)=r=s(\gamma)$.  We will denote by  $\mathcal E_{n}$  the set of empty pairs in $\G_n\times Q_0$.\\

\medskip

\begin{teo}\cite[Corollary 1]{RR15}\label{TeoRR} Let $A=\K Q/I$ be a gentle algebra. Then
\begin{itemize}
\item[(a)] $\HH{0}=1+ \abs{^-(Q_0 // \mathcal P_1)^-}$.
\item[(b)] $\HH{1}=1+\abs{Q_1}-\abs{Q_0}+\abs{\Par{1}}+ \begin{cases} \abs{(Q_1//Q_0)} & \mbox{if }
\operatorname{char}(k)=2\\ 0& \mbox{otherwise} \end{cases}$.
\item[(c)] $\HH{n}=\abs{\Par{n}}+\abs{\mathcal E_n}+a_n\Ge{n}+b_n\Ge{n-1}$ for $n\geq 2$, where $$(a_n,b_n)=
\begin{cases} (1,0)& \mbox{if } \operatorname{char}k\neq 2 \mbox{ and $n$ is even}\\ (0,1)&\mbox{if } \operatorname{char}k\neq 2 \mbox{ and $n$ is odd}\\
(1,1)&\mbox{if } \operatorname{char}k=2 \end{cases} $$
\end{itemize}

\end{teo}

\medskip

\begin{proof}[Proof of Theorem \ref{TeoLad}]


%
%
%
%
%

By definition $\phi_A(0,d)$ is the number of oriented cycles of length $d$ in which each pair of consecutive arrows form a relation. In particular $\phi_A(0,1)$ is the number of loops; i.e. arrows parallels to a vertex. Then  $\phi_A(0,1)=\abs{(Q_1//Q_0)}$.

For $n\geq 2$ let  $\a_1\cdots\a_d$ be a cycle of length $d$ as above with $d\mid n$, then $(\a_1\cdots\a_d)^{\frac{n}{d}}\in \mathcal G_n$. Moreover, if $\rho\in \mathcal G_n$ the classes of $\rho$ and $ t^i(\rho)$ are equal in the quotient $\K \mathcal G_{n}/ \operatorname{Im}(1-t)$, for   $i\in \{1,\cdots, n\}$.  In consequence $\Ge{n}=\sum_{d\mid n} \phi_A(0,d)=\psi_A(n)$.

The fact that the pair $(1,h)$ arises in the algorithm of $\phi_A$ means that the initial permitted thread $H_0$ is equal to the final permitted thread $H_1$ and the forbidden thread $\Pi_0$ which ends in $t(H_0)$ and starts in $s(H_1)$ has length $h$.

If $h=0$ that means that $\Pi_0=e_v$ for some vertex $v$, then  $H_0=H_1$
 is an oriented cycle $\a_1\cdots\a_n\in \mathcal P$.  Moreover,  $e_v$ is a trivial forbidden thread if $v\in Q_0$ with  $\abs{\{\a\in Q_1 \mid s(\a)=v \}}\leq 1$, $\abs{\{\a\in Q_1 \mid t(\a)=v \}}\leq 1$ and if $\b,\gamma \in Q_1$ are such that $s(\gamma)=v=t(\b)$ then $\gamma\b\in I$. Therefore, $\a_n\a_1\in I$ and $(e_v,H_0)\in ^-(Q_0 // \mathcal P_1)^-$. Finally $\phi_A(1,0)=\abs{^-(Q_0 // \mathcal P_1)^-}$.

 If $h=1$ we have $\Pi_0=\a$ for some  $\a\in Q_1$. Then $H_0=H_1=\gamma\in \mathcal P$ with $s(\gamma)=s(\a)$ and $t(\gamma)=t(\a)$.
 If there is $\b\in Q_1$ such that $s(\b)=t(\gamma)$,   since $\gamma$ is a permitted thread, $\gamma\b\in I$. Then $\gamma Q_1\subset I$. Analogously we can see that $ Q_1\gamma \subset I$. Therefore $(\a,\gamma)\in \Par{1}$ and   $\phi_A(1,1)=\abs{\Par{1}}$.

 If $h\geq 2$, $\Pi_0=\a_1\cdots\a_h\in \G_h$. Then $\Pi_0$ can be an oriented cycle or not.  Assume that $\Pi_0$ is an oriented cycle. Then  $H_0=H_1=e_r$ for some vertex $r$. But $e_r$ is a trivial permitted thread if $r\in Q_0$ with  $\abs{\{\a\in Q_1 \mid s(\a)=r \}}\leq 1$, $\abs{\{\a\in Q_1 \mid t(\a)=r \}}\leq 1$ and if $\b,\gamma \in Q_1$ are such that $s(\gamma)=r=t(\b)$ then $\gamma\b\notin I$. Therefore $\a_h\a_1\notin I$ and $(\Pi_0, e_r)$ is an empty pair. Now assume that  $\Pi_0$ is not an oriented cycle. Then  $H_0=H_1=\gamma \in \mathcal P$ with $s(\gamma)=s(\a_1)$ and $t(\gamma)=t(\a_h)$. The condition $\varepsilon (H_0)=-\varepsilon(\Pi_0)$ implies that $\gamma \notin \K Q\alpha_n$. In fact, if $\gamma \in \K Q\alpha_n$, $\varepsilon(\a_n)=\varepsilon (H_0)=-\varepsilon(\gamma)=-\varepsilon(\a_n)$ but $\varepsilon$ takes values on the set $\{1,-1\}$. Analogously,  the condition $\sigma (H_1)=-\sigma(\Pi_0)$ implies that $\gamma \notin \alpha_1 \K Q$. If there is $\b\in Q_1$ such that $s(\b)=t(\gamma)$,   since $\gamma$ is a permitted thread, $\gamma\b\in I$. Then $\gamma Q_1\subset I$. Analogously we can see that $ Q_1\gamma \subset I$ and, in consequence,  $(\Pi_0,\gamma)\in \Par{h}$. We conclude that  $\phi_A(1,h)=\abs{\Par{h}}+\abs{\mathcal E_h}$.

 \end{proof}

\medskip

\section{Main result}

According to \cite[Theorem 6.4]{Gub1} any $m$-cluster tilted algebra of type $\tilde{\mathbb{A}}$  is derived equivalent to  an  $\tilde{\mathbb{A}}$-branched algebra with \textit{normal form} (See \cite[Definition 4.1]{Gub1}).


\begin{figure}[H]

$$\SelectTips{eu}{10}\xymatrix@C=.1pc@R=.2pc{  &&& &&& &&& &&& &&& &&& & .\ar[dddl] \ar@{.}@/^/[r]& . &\\
 &&& &&& &&& &&& &&& &&& &&&\\
 &&& &&& &&& &&& &&& &&& &&&\\
&&& .\ar[rrr]^{\a_2} & \ar@{.}@/^/[lld] & &    \ar@{.}[rrr] &&\ar@{.}@/_/[rr]& . \ar[rrr]^{\a_{r+1}} &&& . \ar@{.}[rrr] &&& . \ar[rrr]^{\a_{s_2}} &&&  . \ar[rrr]_{\a_{s_2+1}} &&& .\ar[uuul] \ar@{.}[rrrdd] &&&  && .\ar[ddll] \ar@{.}@/^/[dr] &&& &&&\\
&&&  &&& &&& &&& &&& &&& &&& &&& &&& . &&& &&& \\
 &&& &&& &&& &&& &&& &&& &&& &&&  . \ar[rrdd]_{\a_{s_2+k_2}}    \\
  &&&  &&& &&& &&& &&& &&& &&& &&& &&& \\
 \scriptstyle{0} \ar[rrruuuu]^{\a_1} \ar[rrrdddd]_{\b_1}  &&& &&& &&& &&& &&& &&& &&&  &&& && .\ar[uuur] \ar[dddr]&&& \\
 &&&  &&& &&& &&& &&& &&& &&& &&& &&& \\
  &&&  &&& &&& &&& &&& &&& &&& &&&   \ar[rruu]^{\b_{s_1+k_1}}  . &&& \\
  &&&  &&& &&& &&& &&& &&& &&&  &&& &&& .  \\
 &&& .\ar[rrr]_{\b_2}  &&& \ar@{.}[rrrrrrrrr] &&&   &&&   &&& . \ar[rrr]_{\b_{s_1}} &&&. \ar[rrr]^{\b_{s_1+1}} &&&  . \ar[dddl] \ar@{.}[rrruu] &&& && . \ar[uull] \ar@{.}@/_/[ur]  &&& &&&\\
&&&  &&& &&& &&& &&& &&& &&&\\
&&&  &&& &&& &&& &&& &&& &&&\\
 &&&  &&& &&& &&& &&& &&& & . \ar[uuul]\ar@{.}@/_/[r]& . &\\
 &&&&&&&&&&&&&&&&&&}$$

\caption{The bound quiver of a normal form.}
\end{figure}
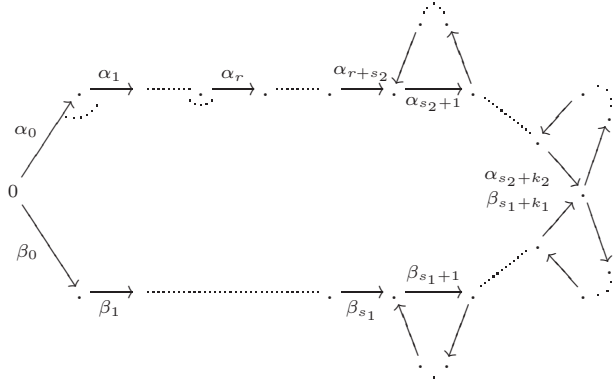

%
%
%

Observe that we allow  $s_1=k_1=0$.\\

Then writing \cite[Proposition 4.2]{Gub1} in terms of the previous parameters we obtain:

\begin{prop}\label{phibranched}
Let $\Lambda$ be an   $\tilde{\mathbb{A}}$-branched algebra with parameters $s_1,s_2,k_1,k_2$ and $r$. Then
$$\phi_{\Lambda}=(mk_1+s_1+r,s_1)^*+ (mk_2+s_2-r,s_2)^*+ (k_1+k_2).(0,m+2)^*$$
\end{prop}

where given a pair $(a,b)\in \mathbb{N} \times \mathbb{N}$,  $(a,b)^\ast$ denotes the characteristic function of the set  $\{(a,b)\} \subseteq \mathbb{N} \times \mathbb{N}$.\qed

%
%
%
%
%

\medskip

We are now able to prove our main result.

\begin{proof}[Proof of Theorem A] Using  Theorem \ref{TeoLad} and Proposition \ref{phibranched}, we compute the dimension of Hochschild cohomology groups for an $m$-cluster tilted algebra of type $\widetilde{\mathbb{A}}$:

\begin{itemize}

\item[(i)] Assume $(mk_1+s_1+r,s_1)=(1,0)$ or $(mk_2+s_2-r,s_2)=(1,0)$. A simple computation gives $s_1=k_1=0$ and $r=1$ or  $s_2=k_2=0$ and $r=-1$ and in both cases $\phi_{\Lambda}(1,0)=1$.

\item[(ii)] The equality $(mk_1+s_1+r,s_1)=(1,1)$ holds if $r=0$,  $ s_1=1$ and   $k_1=0$. Analogously  $(mk_2+s_2-r,s_2)=(1,1)$ if $r=0$,  $s_2=1$ and  $k_2=0$.  Then $\phi_{\Lambda}(1,1)=2$  if both equalities hold in which case $ s_1=s_2=1$ and   $k_1=k_2=0$ or $\phi_{\Lambda}(1,1)=1$ if just one equality holds.

\item[(iii)] Since $m\geq1$ there are no pairs $(0,1)$ and therefore $\phi_{\Lambda}(0,1)=0$.

\item[(iv)]  Assume  $(mk_1+s_1+r,s_1)=(1,n)$ or  $(mk_2+s_2-r,s_2)=(1,n)$ with $n\geq 2$. The first equality never holds and the second one holds if and only if $k_2=0$ and $r+1=n=s_2$. An algebra with those parameters is not an  $m$-cluster tilted of type $\tilde{\mathbb{A}}$ for $m\geq 1$.  Then $\phi_{\Lambda}(1,n)=0$ for $n\geq 2$.

\item[(v)] Since $\phi_{\Lambda}(0,d)= k_1+k_2$ if $d=m+2$ and $0$ otherwise, the function $\psi_{\Lambda}(n)$, defined in Theorem \ref{TeoLad}, depends on $n\equiv 0\pmod{m+2}$, then $$\psi_{\Lambda}(n)=\begin{cases} k_1+k_2 &
    \mbox{if } n\equiv 0\pmod{m+2}\\ 0 & \mbox{otherwise.} \end{cases}$$

Therefore the final expression of the $\Hh{n}$ depends on the value of $n$ and $n-1$ modulo $m+2$, the characteristic of $\K$ and the parity of $n$. We start assuming that $\operatorname{char} \K \neq 2$. If $n\equiv 0 \pmod{m+2}$, we have $\psi_{\Lambda}(n)= k_1+k_2$ and $\psi_{\Lambda}(n-1)=0$, then

\begin{equation*}  \Hh{n}= \phi_{\Lambda}(1,n)+ \begin{cases} k_1+k_2 &  \text{if } n \mbox{ is  even }\\ 0 & \text{if } n \mbox{ is   odd} \end{cases} \end{equation*}

If $n\equiv 1 \pmod{m+2}$, we have $\psi_{\Lambda}(n-1)= k_1+k_2$ and $\psi_{\Lambda}(n)=0$, then

\begin{equation*}  \Hh{n}= \phi_{\Lambda}(1,n)+ \begin{cases} k_1+k_2 &  \text{if } n \mbox{ is  odd }\\ 0 & \text{if } n \mbox{ is   even} \end{cases} \end{equation*}

Since the parity of $n$ can be described in terms of the value of $n$ module $2$, $\Hh{n}$ depends on the value of $n$ modulo $\rm lcm(m+2,2)$. Therefore,  $$\Hh{n}= \phi_{\Lambda}(1,n)+ \begin{cases} k_1+k_2 & \mbox{ if \ } n\equiv 0,1 \pmod{\rm lcm(m+2,2)}; \\ 0 &
\mbox{otherwise.} \end{cases}$$

Finally if $\operatorname{char} \K=2$, the parity of $n$ is not important and  $\Hh{n}$ depends only   on  the value of $n$ module
$m+2$. Then:

$$\Hh{n}=\phi_{\Lambda}(1,n) + \begin{cases} k_1+k_2 & \mbox{ if $n\equiv 0,1\pmod {m+2} $}\\ 0 & \mbox{ otherwise.} \end{cases}$$

\end{itemize}
\end{proof}

\section{Non-trivial structure of the Gerstenhaber algebra $\operatorname{HH}^*(\Lambda)$}

Finally we show that the non-trivial structure of the Gerstenhaber algebra $\operatorname{HH}^*(\Lambda)$ depends only on the existence of $m$-saturated cycles. 
%

\begin{teo}\cite[Theorems 3 and 4]{RR15}\label{gerstenhaberstructure}
Let $A=\K Q/I$ be a gentle algebra such that $\mathcal G_{n}$ is not empty for some $n>0$. Then the cup product defined in $\operatorname{HH}^*(A)$   is non-trivial. If in addition $char\K=0$, then the Lie bracket is also non-trivial.
\end{teo}

\begin{proof}[Proof of Theorem B] By the previous theorem 
it is enough to observe that the set $\mathcal G_{n}$ is not empty if and only if $n\equiv 0\pmod{m+2}$ and $\Lambda$ contains at least one $m$-saturated cycle.

\end{proof}

\section{Further consequences}

We conclude this work showing that the dimension of the Hochschild cohomology groups, together with the number of vertices of the ordinary quiver, do not yield a complete system of invariants  for $\tilde{\mathbb{A}}$-branched algebras. In \cite[Theorem 1.2]{BG14} the authors  showed that any two  $m$-cluster tilted algebras $A$ and $B$ of type $\mathbb{A}$  are derived equivalent if and only if  $\operatorname{HH}^*(A)\cong \operatorname{HH}^*(B)$ and $\abs{Q_0(A)}=\abs{Q_0(B)}$. However, the following example show that the dimension of the Hochschild cohomology loss information.

\medskip

\begin{ejem}
 Let $(Q,I)$ and $(Q',I')$ be the bound quivers
\medskip
\begin{center}
\begin{tabular}{ccccc}
$\SelectTips{eu}{10}\xymatrix@R=1pc@C=.6pc{
&&.\ar[drr]	&	&	& &\\
.\ar[urr] \ar[dr]	&&&&   .\ar[dl]	&& \\
&.\ar[rr]^{\a_0}	&& . \ar[d]^{\a_1}	&	&  & \\
& . \ar[u]^{\a_3} & & .\ar[ll]^{\a_2}&  &   &  }$    & && &     $\SelectTips{eu}{10}\xymatrix@R=1pc@C=.6pc{
&&.\ar[drr]	&	&	& &\\
.\ar[urr] \ar[dr]	&&&&   .\ar[dl]_{\a_0}	&& \\
&.\ar[rr]	&& . \ar[dr]_{\a_1}	&	&  . \ar[ul]_{\a_3}& \\
&  & && . \ar[ur]_{\a_2} & &  }$ \\
&&&&\\

$I=\langle \a_i\a_{i+1}| 0\leqslant i \leqslant 3\rangle$  &&&&  $I'=\langle \a_i\a_{i+1}| 0\leqslant i \leqslant 3 \rangle$  \\
and indices are read modulo 3.&&&& and indices are read modulo 3.
   \end{tabular}
\end{center}

\medskip

\end{ejem}

The algebras $A=\K Q/I$ an $B=\K Q'/I'$ are both $2$-cluster tilted algebras of type $\tilde{\mathbb{A}}$ not derived equivalents because $\phi_A=(3,1)^*+(3,3)^*+(0,4)^*$ and $\phi_B=(2,2)^*+(4,2)^*+(0,4)^*$. However $\abs{Q_0(A)}=\abs{Q_0(B)}$ and $\operatorname{dim}_{\K}\operatorname{HH}^i(A)=\operatorname{dim}_{\K} \operatorname{HH}^i(B)$ for all $i$.

\medskip

\section*{Acknowledgements}
The author gratefully  acknowledges financial support from  the  \emph{Agencia Nacional de Investigaci\'{o}n e Innovaci\'{o}n (ANII)} of Uruguay.

\medskip

\end{document}